\documentclass[11pt,a4paper,reqno]{amsart}

\usepackage{latexsym}
\usepackage{amsmath}
\usepackage{amsfonts}
\usepackage{a4wide}
\usepackage{amssymb}
\usepackage{amsthm}
\usepackage[colorlinks, citecolor=blue, linkcolor=red]{hyperref}

\theoremstyle{plain}

\newtheorem{teo}{Theorem}[section]
\newtheorem{lemma}[teo]{Lemma}
\newtheorem{prop}[teo]{Proposition}
\newtheorem{cor}[teo]{Corollary}
\newtheorem{ackn}{Acknowledgments\!}

\theoremstyle{definition}

\theoremstyle{remark}

\newtheorem{rem}[teo]{Remark}

\numberwithin{equation}{section}

\def\rd{\overset{\circ}{Ric}}
\def\rdc{\overset{\circ}{R}}
\def\kn{\mathbin{\bigcirc\mkern-15mu\wedge}}

\def\SS{{{\mathbb S}}}

\def\gt{\widetilde{g}}

\title[Integral pinched shrinking Ricci solitons]{Integral pinched shrinking Ricci solitons}

\author[Giovanni Catino]{Giovanni Catino}
\address[Giovanni Catino]{Dipartimento di Matematica, Politecnico di Milano, Piazza Leonardo da Vinci 32, 20133 Milano, Italy}
\email[]{giovanni.catino@polimi.it}

\date{\today}

\begin{document}

\begin{abstract} 

We prove that a $n$-dimensional, $4 \leq n \leq 6$, compact gradient shrinking Ricci soliton satisfying a $L^{n/2}$-pinching condition is isometric to a quotient of the round $\mathbb{S}^{n}$. The proof relies mainly on sharp algebraic curvature estimates, the Yamabe-Sobolev inequality and an improved rigidity result for integral pinched Einstein metrics.
\end{abstract}

\maketitle

\begin{center}


\medskip

\centerline{\bf AMS subject classification:  53C24, 53C25}

\end{center}

\

\section{Introduction}

In this paper we investigate compact gradient shrinking Ricci solitons satisfying a $L^{n/2}$-pinching condition. We recall that a Riemannian manifold $(M^n ,g)$ is a {\em gradient Ricci soliton} if there exists a smooth function $f$ on $M^n$ such that 
$$
Ric + \nabla^2 f \, = \, \lambda g
$$ 
for some constant $\lambda$. The Ricci soliton is called {\em shrinking} if $\lambda>0$, {\em steady} if $\lambda=0$ and {\em expanding} if $\lambda<0$. The soliton is {\em trivial} if $\nabla f$ is parallel. Ricci solitons generate self-similar solutions of the Ricci flow, play a fundamental role in the formation of singularities and have been studied by several authors (see H.-D. Cao~\cite{cao2} for a nice overview). 

We will focus only on compact Ricci solitons. It has been shown by G. Perelman~\cite{perel1} that there is no non-trivial compact steady and expanding Ricci soliton. In dimension three, T. Ivey~\cite{ivey1} proved that the only compact shrinking Ricci solitons are quotients of $\SS^{3}$ with its standard metric. Dimension four is the lowest dimension where there are interesting examples of shrinking Ricci solitons. The first examples where constructed by N. Koiso~\cite{koiso1} and H.-D. Cao~\cite{cao1} (see also~\cite{ilman6, wangzhu}). Note that all of the known interesting compact shrinking Ricci solitons are K\"ahler. 

In the last years there have been a lot of interesting results concerning the classification of compact shrinking Ricci solitons satisfying {\em pointwise} curvature conditions. For instance, it follows by the work of C. B\"ohm and B. Wilking~\cite{bohmwilk} that the only compact shrinking Ricci solitons with positive (two-positive) curvature operator are quotients of $\SS^{n}$. We have the same conclusion also in the locally conformally flat case~\cite{mantemin2}. Other rigidity results on compact shrinkers can be found in~\cite{fergarciario2, sesummunte, caochen2, cat2, wuwuwylie, xcaotran}.

In this paper we will show that quotients of $\SS^{n}$ are the only compact shrinking solitons satisfying  an {\em integral} curvature condition. We observe that all the results apply also to complete (possibly non-compact) gradient shrinking Ricci solitons with positive sectional curvature, since O. Munteanu and J. Wang in~\cite{munwang} recently showed that these conditions force the manifold to be compact.

To fix the notation, we will denote by $W$, $\rd$ and $R$ the Weyl, traceless Ricci and scalar curvature respectively. $Y(M,[g])$ will be the Yamabe invariant associated to $(M^{n},g)$ and the norm of a $(k,l)$-tensors $T$ is defined as $
|T|^{2}_{g}=g^{i_{1}m_{1}}\cdots g^{i_{k}m_{k}}
 g_{j_{1}n_{1}}\dots g_{j_{l}n_{l}} T_{i_{1}\dots i_{k}}^{j_{1}\dots j_{l}} T_{m_{1}\dots m_{k}}^{n_{1}\dots n_{l}}\,.
$

\medskip

In dimension four we will prove the following $L^{2}$-pinching result.

\begin{teo}\label{teo-4d1}  Every four-dimensional compact shrinking Ricci soliton satisfying
\begin{equation*} 
\int_{M}|W|^{2}\,dV_{g}+\int_{M}|\rd|^{2}\,dV_{g} \,<\, \frac{1}{48}Y(M,[g])^{2}
\end{equation*}
is isometric to a quotient of the round $\mathbb{S}^{4}$. 
\end{teo}

As a corollary, using a lower bound for the Yamabe invariant proved in~\cite{gursky1}, we get the following result.

\begin{cor}\label{cor-4d2}  Every four-dimensional compact shrinking Ricci soliton satisfying
\begin{equation*} 
\int_{M}|W|^{2}\,dV_{g}+\frac{5}{4}\int_{M}|\rd|^{2}\,dV_{g} \,\leq\, \frac{1}{48}\int_{M}R^{2}\,dV_{g}
\end{equation*}
is isometric to a quotient of the round $\mathbb{S}^{4}$. 
\end{cor}
\begin{rem}\label{rem1}
The pinching condition in Corollary~\ref{cor-4d2} is equivalent to the following (see Section~\ref{sec-4d})
$$
\int_{M}|W|^{2}\,dV_{g}+\frac{2}{39}\int_{M}R^{2}\,dV_{g} \,\leq\,\frac{160}{13}\pi^{2}\chi(M)\,,
$$
where $\chi(M)$ is the Euler-Poincar\'e characteristic of $M$.
\end{rem}

We notice that in Corollary~\ref{cor-4d2} and Remark~\ref{rem1} we only have to assume the weak inequality. In fact, when equality occurs, we can show that $(M^{n},g)$ has to be conformally Einstein, in particular Bach-flat, and using~\cite{caochen2} the conclusion follows (see Section~\ref{sec-4d}).

Theorem~\ref{teo-4d1} is the four-dimensional case of the following result which holds in every dimension $4\leq n \leq 6$. 

%

\begin{teo}\label{teo-ndim} Let $(M^{n},g)$ be a $n$-dimensional, $4 \leq n \leq 6$, compact shrinking Ricci solitons satisfying
\begin{equation*} 
\left(\int_{M}\Big|W+\frac{\sqrt{2}}{\sqrt{n}(n-2)}\rd\kn g\Big|^{n/2}dV_{g}\right)^{\frac{2}{n}}+\sqrt{\frac{(n-4)^{2}(n-1)}{8(n-2)}}\lambda V(M)^{\frac{2}{n}} < \sqrt{\frac{n-2}{32(n-1)}}Y(M,[g]) \,.
\end{equation*}
Then $(M^{n},g)$ is isometric to a quotient of the round $\mathbb{S}^{n}$. Moreover, in dimension $5\leq n \leq 6$, the same result holds only assuming the weak inequality.
\end{teo}
It is easy to see that, on a shrinking Ricci soliton of dimension $n\geq 7$, the pinching condition does not hold, since
$$
\lambda V(M)^{\frac{2}{n}} \,=\, \frac{1}{n}V(M)^{\frac{2-n}{n}}\int_{M}R\,dV_{g}\,\geq\,\frac{1}{n}Y(M,[g]) \,.
$$ 
The proof of Theorem~\ref{teo-ndim} is inspired by the classification of Einstein (or locally conformally flat) metrics satisfying a $L^{n/2}$-pinching condition (see~\cite{hebvau1, hebvau2, gursky1} and also~\cite{bour}). More precisely, we will use the soliton equation to obtain an elliptic PDE on $|\rd|^{2}$. Since every compact shrinking soliton has positive scalar curvature, the positivity of the Yamabe invariant $Y(M,[g])$ implies a Sobolev-type inequality on $|\rd|$ which, combined with the PDE, allows us to get a $L^{n/2}$-estimate on the curvature on every non-Einstein shrinking Ricci solitons. In doing this, we will prove an algebraic curvature estimate (see Proposition~\ref{prop-alg}) which holds in every dimension and was firstly observed in dimension four in~\cite{bour}. The pinching assumption of Theorem~\ref{teo-ndim} implies that the manifold $(M^{n},g)$ has to be Einstein. To get the final result, one has to show that these Einstein metrics are actually of constant positive sectional curvature. It turns out that we cannot directly apply the result of E. Hebey and M. Vaugon~\cite[Theorem 1]{hebvau1}. On the other hand, as already observed in their paper, one can improve the estimates using sharper algebraic inequalities on the Weyl curvature (see Lemma~\ref{lem-tachi}). To this aim, we will prove the following rigidity result for positively curved Einstein manifolds, which improves~\cite[Theorem 1]{hebvau1} in dimension $4\leq n \leq 6$.

\begin{teo}\label{teo-ein0} Let $(M^{n},g)$ be a $n$-dimensional Einstein manifold with positive scalar curvature. There exists a positive constant $A(n)$ such that if
$$
\left(\int_{M}|W|^{\frac{n}{2}}\,dV_{g}\right)^{\frac{2}{n}} \,<\, A(n)\,Y(M,[g]) \,,
$$
then $(M^{n},g)$ is isometric to a quotient of the round $\SS^{n}$.We can take $A(4)=\frac{5}{9\sqrt{6}}$, $A(5)=\frac{3}{32}$, $A(6)=\frac{\sqrt{3}}{5\sqrt{70}}$, $A(n)=\frac{n-2}{20(n-1)}$ if $7\leq n \leq 9$ and $A(n)=\frac{2}{5n}$ if $n\geq 10$.
\end{teo}

%
%
%
%

\

\section{Algebraic preliminaries} 

To fix the notation we recall that the Riemann curvature operator of a Riemannian manifold $(M^n,g)$ is defined as in~\cite{gahula} by
$$
Rm(X,Y)Z=\nabla_{Y}\nabla_{X}Z-\nabla_{X}\nabla_{Y}Z+\nabla_{[X,Y]}Z\,.
$$ 
In a local coordinate system the components of the $(3,1)$--Riemann curvature tensor are given by
$R^{l}_{ijk}\frac{\partial}{\partial
  x^{l}}=Rm\big(\frac{\partial}{\partial
  x^{i}},\frac{\partial}{\partial
  x^{j}}\big)\frac{\partial}{\partial x^{k}}$ and we denote by
$R_{ijkl}=g_{lp}R^{p}_{ijk}$ its $(4,0)$--version. Throughout the paper the Einstein convention of summing over the repeated indices will be adopted. The Ricci tensor $Ric$ is obtained by the contraction $(Ric)_{ik}=R_{ik}=g^{jl}R_{ijkl}$, $R=g^{ik}R_{ik}$ will denote the scalar curvature and $(\rd)_{ik}=R_{ik}-\frac{1}{n}R\, g_{ik}$ the traceless Ricci tensor. The Weyl tensor $W$ is then defined by the following orthogonal decomposition formula (see~\cite[Chapter~3,
Section~K]{gahula}) in dimension $n\geq 3$,
\begin{eqnarray*}
W \,=\, Rm - \frac{1}{n-2}\rd\kn g - \frac{1}{2n(n-1)}R \, g\kn g\,
\end{eqnarray*} 
where $\kn$ denotes the Kulkarni-Nomizu product which is defined as follow: let $A, B$ two symmetric $(0,2)$-tensors, then 
$$
(A\kn B)_{ijkl}=A_{ik}B_{jl}-A_{il}B_{jk}-A_{jk}B_{il}+A_{jl}B_{ik}.
$$ 
The Riemannian metric induces norms on all the tensor bundles, in coordinates this norm is given, for a tensor $T=T_{i_{1}\dots i_{k}}^{j_{1}\dots j_{l}}$, by
$$
|T|^{2}_{g}=g^{i_{1}m_{1}}\cdots g^{i_{k}m_{k}}
 g_{j_{1}n_{1}}\dots g_{j_{l}n_{l}} T_{i_{1}\dots i_{k}}^{j_{1}\dots j_{l}} T_{m_{1}\dots m_{k}}^{n_{1}\dots n_{l}}\,.
$$

In order to prove Theorem~\ref{teo-ndim}, we will use two algebraic curvature inequalities. The first one, which involves the Weyl curvature and the traceless-Ricci curvature can be viewed as a combination of the algebraic inequality for trace-free symmetric two-tensors $T=\{T_{ij}\}$ (for a proof see~\cite[Lemma 2.4]{huisk9})
\begin{equation}\label{eq-oku}
\left|T_{ij}T_{jk}T_{ik}\right| \,\leq \, \sqrt{\frac{n-2}{n(n-1)}} |T|^{3} \,
\end{equation}
with Huisken inequality~\cite[Lemma 3.4]{huisk9}
\begin{equation}\label{eq-hui}
 \left|W_{ijkl}\rdc_{ik}\rdc_{jl}\right| \,\leq \, \sqrt{\frac{n-2}{2(n-1)}} |W||\rd|^{2} \,.
\end{equation}
The following estimates was proved in dimension four in~\cite[Lemma 4.7]{bour}.
\begin{prop}\label{prop-alg} 
On every $n$-dimensional Riemannian manifold the following estimate holds
$$
\left|-W_{ijkl}\rdc_{ik}\rdc_{jl}+\frac{2}{n-2}\rdc_{ij}\rdc_{jk}\rdc_{ik}\right| \,\leq \, \sqrt{\frac{n-2}{2(n-1)}}\left(|W|^{2}+\frac{8}{n(n-2)}|\rd|^{2}\right)^{1/2} |\rd|^{2} \,.
$$
\end{prop}
\begin{proof} We follow~\cite[Lemma 4.7]{bour}. First of all we have
\begin{align*}
(\rd \kn g)_{ijkl}&= (\rdc_{ik}g_{jl}-\rdc_{il}g_{jk}+\rdc_{jl}g_{ik}-\rdc_{jk}g_{il})\\
(\rd\kn\rd)_{ijkl} &= 2(\rdc_{ik}\rdc_{jl}-\rdc_{il}\rdc_{jk})\,.
\end{align*}
An easy computation shows
\begin{align*}
W_{ijkl}\rdc_{ik}\rdc_{jl} &= \frac{1}{4} W_{ijkl}(\rd\kn\rd)_{ijkl}\\
\rdc_{ij}\rdc_{jk}\rdc_{ik} &= -\frac{1}{8} (\rd \kn g)_{ijkl} (\rd\kn\rd)_{ijkl} \,.
\end{align*}
Hence we get the following identity
\begin{equation}\label{alg1}
-W_{ijkl}\rdc_{ik}\rdc_{jl}+\frac{2}{n-2}\rdc_{ij}\rdc_{jk}\rdc_{ik} \,=\, -\frac{1}{4} \left(W+\frac{1}{n-2}\rd\kn g\right)_{ijkl}(\rd\kn\rd)_{ijkl} \,.
\end{equation}
Since $\rd\kn\rd$ has the same symmetries of the Riemann tensor, it can be orthogonally decomposed as
$$
\rd\kn\rd \,=\, T + V + U 
$$
where $T$ is totally trace-free and 
$$
V_{ijkl} \,=\,-\frac{2}{n-2}(\rd\,^{2}\kn g)_{ijkl}+\frac{2}{n(n-2)}|\rd|^{2}(g\kn g)_{ijkl}
$$
$$
U_{ijkl} \,=\, -\frac{1}{n(n-1)}|\rd|^{2} (g\kn g)_{ijkl} \,,
$$
where $(\rd\,^{2})_{ik}=\rdc_{ip}\rdc_{kp}$. Taking the squared norm one obtains
\begin{align*}
|\rd\kn\rd|^{2} &= 8|\rd|^{4}-8|\rd\,^{2}|^{2}\\
|V|^{2} &= \frac{16}{n-2}|\rd\,^{2}|^{2}-\frac{16}{n(n-2)}|\rd|^{4} \\
|U|^{2} &= \frac{8}{n(n-1)}|\rd|^{4}\,.
\end{align*}
In particular, one has
\begin{equation*}\label{eqalg1}
|T|^{2}+\frac{n}{2}|V|^{2} \,=\, |\rd\kn\rd|^{2}+\frac{n-2}{2}|V|^{2}-|U|^{2} \,=\, \frac{8(n-2)}{n-1}|\rd|^{4}\,.
\end{equation*}
We now estimate the right hand side of~\eqref{alg1}. Using the fact that $W$ and $T$ are totally trace-free and the Cauchy-Schwarz inequality we obtain
\begin{align*}
\left|\left(W+\frac{1}{n-2}\rd\kn g\right)_{ijkl}(\rd\kn\rd)_{ijkl}\right|^{2} &= \left|\left(W+\frac{1}{n-2}\rd\kn g\right)_{ijkl}(T+V)_{ijkl}\right|^{2} \\
&= \left|\left(W+\frac{\sqrt{2}}{\sqrt{n}(n-2)}\rd\kn g\right)_{ijkl}\left(T+\sqrt{\frac{n}{2}}V\right)_{ijkl}\right|^{2}\\
&\leq \left|\left(W+\frac{\sqrt{2}}{\sqrt{n}(n-2)}\rd\kn g\right)\right|^{2}\left(|T|^{2}+\frac{n}{2}|V|^{2}\right) \\
&= \frac{8(n-2)}{n-1}\left(|W|^{2}+\frac{8}{n(n-2)}|\rd|^{2}\right)|\rd|^{4}\,.
\end{align*}
This estimate together with~\eqref{alg1} concludes the proof.

\end{proof} 

\begin{rem} Notice that this estimate is stronger than the one obtained by using triangle inequality together with inequalities~\eqref{eq-oku} and~\eqref{eq-hui}.
\end{rem}

We conclude this section with a second algebraic inequality on the Weyl tensor. Let $T=\{T_{ijkl}\}$ be a tensor with the same symmetries as the Riemann tensor. It defines a symmetric operator $T:\Lambda^{2}(M)\longrightarrow \Lambda^{2}(M)$ on the space of two-forms by
$$
(T \omega)_{kl} \,:=\,\frac{1}{2}T_{ijkl}\omega_{ij}\,,
$$
with $\omega\in\Lambda^{2}(M)$. Hence we have that $\mu$ is an eigenvalue of $T$ if $T_{ijkl}\omega_{ij} = 2\mu\,\omega_{kl}$, for some $0\neq \omega\in\Lambda^{2}(M)$ and we have $\Vert T\Vert^{2}_{\Lambda^{2}}=\frac{1}{4}|T|^{2}$.
\begin{lemma}\label{lem-tachi} On every $n$-dimensional Riemannian manifold there exists a positive constant $C(n)$ such that the following estimate holds
$$
2W_{ijkl}W_{ipkq}W_{pjql}+\frac{1}{2}W_{ijkl}W_{klpq}W_{pqij} \,\leq \, C(n)\,|W|^{3} \,.
$$
We can take $C(4)=\frac{\sqrt{6}}{4}$, $C(5)=1$, $C(6)=\frac{\sqrt{70}}{2\sqrt{3}}$ and $C(n)=\frac{5}{2}$ for $n\geq 7$.
\end{lemma}
\begin{proof} 
First of all, simply by Cauchy-Schwartz, in every dimension one has 
$$
2W_{ijkl}W_{ipkq}W_{pjql}+\frac{1}{2}W_{ijkl}W_{klpq}W_{pqij} \,\leq \, \frac{5}{2}|W|^{3}\,.
$$
In dimension $n=4$, the sharper constant $C(4)=\frac{\sqrt{6}}{4}$ was proved in~\cite[Lemma 3.5]{huisk9}. 

To get the estimate in dimension $n=5$, we invoke the following algebraic identity which holds on every five-dimensional Riemannian manifolds (for instance see~\cite[Equation (A.3)]{jackparker})
$$
W_{ijkl}W_{klpq}W_{pqij} \,=\, 4W_{ijkl}W_{ipkq}W_{pjql} \,.
$$
Hence, by Cauchy-Schwartz, in dimension $n=5$ one has
$$
2W_{ijkl}W_{ipkq}W_{pjql}+\frac{1}{2}W_{ijkl}W_{klpq}W_{pqij} \,=\, W_{ijkl}W_{klpq}W_{pqij} \,\leq\, |W|^{3}\,.
$$

Finally, to obtain the estimate in dimension $n=6$ we proceed as in~\cite{huisk9} using an idea of S. Tachibana~\cite{tachib1} and define for some fixed indices $p,q,r,s$ the local skew symmetric tensor $\omega=\{\omega_{ij}^{(pqrs)}\}$ by
\begin{align*}
\omega_{ij}^{(pqrs)} :=& W_{iqrs}g_{jp}+W_{pirs}g_{jq}+W_{pqis}g_{jr}+W_{pqri}g_{js}\\
&\,-W_{jqrs}g_{ip}-W_{pjrs}g_{iq}-W_{pqjs}g_{ir}-W_{pqrj}g_{is}\,.
\end{align*}
Then a simple computation shows that
\begin{equation}\label{eq-alg7}
2W_{ijkl}W_{ipkq}W_{pjql}+\frac{1}{2}W_{ijkl}W_{klpq}W_{pqij} \,=\, \frac{1}{16}W_{ijkl}\,\omega_{ij}^{(pqrs)}\omega_{kl}^{(pqrs)}
\end{equation}
$$
|\omega|^{2} \,=\, 8(n-1) |W|^{2} \,.
$$
Now if we denote by $\mu$ the maximum eigenvalue of $W$, since $W$ is trace-free, it follows from~\cite[Corollary 2.5]{huisk9} that
$$
W_{ijkl}\,\omega_{ij}^{(pqrs)}\omega_{kl}^{(pqrs)} \,\leq\, 2 \mu |\omega|^{2} \,\leq \, \sqrt{\frac{(n-2)(n+1)}{n(n-1)}}|W||\omega|^{2} \,=\,  8(n-1)\sqrt{\frac{(n-2)(n+1)}{n(n-1)}}|W|^{3}\,
$$
and we obtain
$$
2W_{ijkl}W_{ipkq}W_{pjql}+\frac{1}{2}W_{ijkl}W_{klpq}W_{pqij} \,\leq\, \frac{1}{2}\sqrt{\frac{(n-1)(n-2)(n+1)}{n}} |W|^{3}\,.
$$
Substituting in dimension $n=6$, we get $C(6)=\frac{\sqrt{70}}{2\sqrt{3}}<\frac{5}{2}$.
\end{proof}

\

\section{Proof of the main theorem}

In this section we will prove Theorem~\ref{teo-ndim}. Let $(M^{n},g)$ be a $n$-dimensional compact gradient shrinking Ricci solitons 
$$
Ric + \nabla^{2} f \,=\, \lambda g 
$$
for some smooth function $f$ and some positive constant $\lambda>0$. First of all we recall the following well known formulas (for the proof see~\cite{mantemin2}) 

\begin{lemma}\label{formulas} Let $(M^{n},g)$ be a gradient Ricci soliton, then the following formulas hold
\begin{equation}\label{eq0}
\Delta f \,=\, n \lambda - R 
\end{equation}
\begin{equation}\label{eq1}
\Delta_{f} R \,= \, 2\lambda R-2|Ric|^2
\end{equation}
\begin{align}
\Delta_{f} R_{ik}
\,=\,&\,2\lambda R_{ik}-2\,W_{ijkl} R_{jl}\label{eq2}\\
&\,+\frac{2}{(n-1)(n-2)}
\bigl(R^2 g_{ik}-n R\,R_{ik}
+2(n-1)R_{ ij} R^{j}_{\,k}-(n-1)|Ric|^{2} g_{ik}\bigr)\,,\nonumber
\end{align}
where the $\Delta_{f}$ denotes the $f$-Laplacian, $\Delta_{f}=\Delta-\nabla_{\nabla f}$.
\end{lemma}

In particular, using $\rdc_{ij}=R_{ij}-\frac{1}{n}R g_{ij}$ and the fact that
$$
R_{ij}R_{jk}R_{ik} \,=\, \rdc_{ij}\rdc_{jk}\rdc_{ik} + \frac{3}{n}R |Ric|^{2} - \frac{2}{n^{2}} R^{3} 
$$
a simple computation shows the following equation for the $f$-Laplacian of the squared norm of the treceless Ricci tensor
\begin{lemma}\label{formulas2} Let $(M^{n},g)$ be a gradient Ricci soliton, then the following formula holds
\begin{equation*}
\frac{1}{2}\Delta_{f} |\rd|^{2} \,=\, |\nabla \rd|^{2} + 2\lambda |\rd|^{2} -2W_{ijkl}\rdc_{ik}\rdc_{jl}+\frac{4}{n-2}\rdc_{ij}\rdc_{jk}\rdc_{ik}-\frac{2(n-2)}{n(n-1)}R|\rd|^{2}\,.
\end{equation*}
\end{lemma}

Moreover, it follows immediately from equation~\eqref{eq1} and the maximum principle that every compact shrinking solitons has positive scalar curvature (see~\cite{ivey1}).

\

We will denote $Y(M,[g])$ the Yamabe invariant associated to $(M^{n},g)$ (here $[g]$ is the conformal class of $g$) defined by
$$
Y(M,[g]) \,=\, \inf_{\gt\in[g]} \frac{\int_{M}\widetilde{R}\,dV_{\gt}}{\left(\int_{M}\,dV_{\gt}\right)^{(n-2)/n}} \,=\, \frac{4(n-1)}{n-2}\inf_{u\in W^{1,2}(M)} \frac{\int_{M}|\nabla u|^{2}\,dV_{g}+\frac{n-2}{4(n-1)}\int_{M}R\,u^{2}\,dV_{g}}{\left(\int_{M}|u|^{2n/(n-2)}\,dV_{g}\right)^{(n-2)/n}}
$$
It is well known that, on a compact manifold, $Y(M,[g])$ is positive (respectively zero or negative) if and only if there exists a conformal metric in $[g]$ with everywhere positive (respectively zero or negative) scalar curvature. Then, every compact shrinking soliton has positive Yamabe invariant $Y(M,[g])>0$, so for every $u\in W^{1,2}(M)$ the following Yamabe-Sobolev inequality holds 
\begin{equation}\label{sobolev}
\frac{n-2}{4(n-1)}Y(M,[g]) \left(\int_{M}|u|^{\frac{2n}{n-2}}\,dV_{g}\right)^{\frac{n-2}{n}} \,\leq \, \int_{M} |\nabla u|^{2}\,dV_{g} +\frac{n-2}{4(n-1)}\int_{M}R\, u^{2} \,dV_{g} \,.
\end{equation}

Now, let us assume that $4\leq n \leq 6$ and $(M^{n},g)$ satisfies the integral pinching condition as in Theorem~\ref{teo-ndim}
\begin{equation}\label{pinching}
\left(\int_{M}\Big|W+\frac{\sqrt{2}}{\sqrt{n}(n-2)}\rd\kn g\Big|^{n/2}dV_{g}\right)^{\frac{2}{n}}+\sqrt{\frac{(n-4)^{2}(n-1)}{8(n-2)}}\lambda V(M)^{\frac{2}{n}} \,<\, \sqrt{\frac{n-2}{32(n-1)}}Y(M,[g])\,,
\end{equation}
with $V(M):=\int_{M}dV_{g}$. From Lemma~\ref{formulas2} we have
\begin{equation*}
\frac{1}{2}\Delta_{f} |\rd|^{2} \,=\, |\nabla \rd|^{2} + 2\lambda |\rd|^{2} -2W_{ijkl}\rdc_{ik}\rdc_{jl}+\frac{4}{n-2}\rdc_{ij}\rdc_{jk}\rdc_{ik}-\frac{2(n-2)}{n(n-1)}R|\rd|^{2}\,.
\end{equation*}
Using Kato inequality ($|\nabla\rd|^{2} \geq |\nabla|\rd||^{2}$ at every point where $|\rd|\neq 0$) and Proposition~\ref{prop-alg} we obtain
\begin{align*}
0 \geq& -\frac{1}{2}\Delta_{f} |\rd|^{2}+|\nabla|\rd||^{2} + 2 \lambda |\rd|^{2} \\
&\, -  \sqrt{\frac{2(n-2)}{n-1}}\left(|W|^{2}+\frac{8}{n(n-2)}|\rd|^{2}\right)^{1/2} |\rd|^{2} -\frac{2(n-2)}{n(n-1)}R|\rd|^{2} \,.
\end{align*}
Integrating by parts over $M^{n}$ and using equation~\eqref{eq0} it follows that
\begin{align*}
0 \geq& -\frac{1}{2}\int_{M}|\rd|^{2}\Delta f\,dV_{g} +\int_{M}|\nabla|\rd||^{2}\,dV_{g} + 2 \lambda \int_{M} |\rd|^{2}\,dV_{g}\\
&\,-  \sqrt{\frac{2(n-2)}{n-1}}\int_{M}\left(|W|^{2}+\frac{8}{n(n-2)}|\rd|^{2}\right)^{1/2} |\rd|^{2}\,dV_{g} -\frac{2(n-2)}{n(n-1)}\int_{M}R|\rd|^{2}\,dV_{g}\\
=& \,\int_{M}|\nabla|\rd||^{2}\,dV_{g} -\frac{n-4}{2} \lambda \int_{M} |\rd|^{2}\,dV_{g}\\
&\,-  \sqrt{\frac{2(n-2)}{n-1}}\int_{M}\left(|W|^{2}+\frac{8}{n(n-2)}|\rd|^{2}\right)^{1/2} |\rd|^{2}\,dV_{g} +\frac{n^{2}-5n+8}{2n(n-1)}\int_{M}R|\rd|^{2}\,dV_{g}
\end{align*}
Using inequality~\eqref{sobolev} with $u:=|\rd|$ we get
\begin{align*}
0 \geq& \,\frac{n-2}{4(n-1)}Y(M,[g]) \left(\int_{M}|\rd|^{\frac{2n}{n-2}}\,dV_{g}\right)^{\frac{n-2}{n}}-\frac{n-4}{2} \lambda \int_{M} |\rd|^{2}\,dV_{g} \\
&\,-\sqrt{\frac{2(n-2)}{n-1}}\int_{M}\left(|W|^{2}+\frac{8}{n(n-2)}|\rd|^{2}\right)^{1/2} |\rd|^{2}\,dV_{g} +\frac{(n-4)^{2}}{4n(n-1)}\int_{M}R|\rd|^{2}\,dV_{g}
\end{align*}
From H\"older inequality, since $\lambda>0$ and $n\geq 4$, we obtain
\begin{align*}
0 \geq& \left(\frac{n-2}{4(n-1)}Y(M,[g]) - \frac{n-4}{2}\lambda V(M)^{\frac{2}{n}} -\sqrt{\frac{2(n-2)}{n-1}}\left(\int_{M}\left(|W|^{2}+\frac{8}{n(n-2)}|\rd|^{2}\right)^{n/4}dV_{g}\right)^{\frac{2}{n}}\right) \\
&\, \left(\int_{M}|\rd|^{\frac{2n}{n-2}}\,dV_{g}\right)^{\frac{n-2}{n}} +\frac{(n-4)^{2}}{4n(n-1)}\int_{M}R|\rd|^{2}\,dV_{g} \,.
\end{align*}
Thus, either $|\rd|\equiv 0$, i.e. $(M^{n},g)$ is Einstein, or the following estimate holds
$$
\left(\int_{M}\left(|W|^{2}+\frac{8}{n(n-2)}|\rd|^{2}\right)^{n/4}dV_{g}\right)^{\frac{2}{n}}+\sqrt{\frac{(n-4)^{2}(n-1)}{8(n-2)}}\lambda V(M)^{\frac{2}{n}} \,\geq\, \sqrt{\frac{n-2}{32(n-1)}}Y(M,[g])\,.
$$
Since $W$ is totally trace-free, one has
$$
\left|W+\frac{\sqrt{2}}{\sqrt{n}(n-2)}\rd\kn g\right|^{2} \,=\,  |W|^{2} +\frac{8}{n(n-2)}|\rd|^{2} 
$$
and the pinching condition~\eqref{pinching} implies that $(M^{n},g)$ is Einstein. Moreover, observe that in dimension $n=5,6$, we get the same conclusion if we assume just the weak inequality in~\eqref{pinching}. Since $g$ is Einstein, by Obata Theorem~\cite{oba} one has that $g$ is a (the) Yamabe metric of $[g]$, i.e.
$$
Y(M,[g]) \,=\, V(M)^{\frac{2-n}{n}}\int_{M}R\,dV_{g}\,.
$$
Moreover, integrating equation~\eqref{eq0} one has
$$
\lambda V(M)^{\frac{2}{n}} \,=\, \frac{1}{n}V(M)^{\frac{2-n}{n}}\int_{M}R\,dV_{g}\,=\,\frac{1}{n}Y(M,[g]) \,.
$$ 
Hence, the pinching condition~\eqref{pinching} implies
\begin{equation}\label{pinchein}
\left(\int_{M}|W|^{\frac{n}{2}}\,dV_{g}\right)^{\frac{2}{n}} \,\leq\, \frac{8n-n^{2}-8}{4n\sqrt{2(n-1)(n-2)}}Y(M,[g])\,.
\end{equation}
To conclude the proof of Theorem~\ref{teo-ndim} we need to show that an Einstein manifold of dimension $4\leq n\leq 6$ with positive scalar curvature and satisfying~\eqref{pinchein} must have constant positive sectional curvature. This is the content of the next theorem. It is sufficient to observe that the pinching constant in~\eqref{pinchein} is strictly smaller than the one in the following theorem whenever $4\leq n\leq 6$. The next result is Theorem~\ref{teo-ein0} in the introduction.

\begin{teo}\label{teo-ein} Let $(M^{n},g)$ be a $n$-dimensional Einstein manifold with positive scalar curvature. There exists a positive constant $A(n)$ such that if
$$
\left(\int_{M}|W|^{\frac{n}{2}}\,dV_{g}\right)^{\frac{2}{n}} \,<\, A(n)\,Y(M,[g]) \,,
$$
then $(M^{n},g)$ is isometric to a quotient of the round $\SS^{n}$.We can take $A(4)=\frac{5}{9\sqrt{6}}$, $A(5)=\frac{3}{32}$, $A(6)=\frac{\sqrt{3}}{5\sqrt{70}}$, $A(n)=\frac{n-2}{20(n-1)}$ if $7\leq n \leq 9$ and $A(n)=\frac{2}{5n}$ if $n\geq 10$.
\end{teo}
\begin{proof} Following the proof in~\cite{hebvau1}, by Bianchi identities and the fact that $g$ is Einstein, one can prove that the Weyl tensor satisfies the following equation (see~\cite[Equation (5)]{hebvau1})
$$
\Delta W_{ijkl} \,=\, \frac{2}{n}R\, W_{ijkl} -2\left(2W_{ipkq}W_{pjql}+\frac{1}{2}W_{klpq}W_{pqij}\right)\,.
$$
Contracting with $W_{ijkl}$, we get
$$
\frac{1}{2}\Delta|W|^{2} \,=\, |\nabla W|^{2} + \frac{2}{n}R|W|^{2}-2\left(2W_{ijkl}W_{ipkq}W_{pjql}+\frac{1}{2}W_{ijkl}W_{klpq}W_{pqij}\right)\,.
$$
Integrating on $M^{n}$, using again Kato inequality and Lemma~\ref{lem-tachi}, one obtains
\begin{align*}
0 \geq& \int_{M}|\nabla|W||^{2}\,dV_{g} +\frac{2}{n}\int_{M}R|W|^{2}\,dV_{g} -2C(n)\int_{M}|W|^{3}\,dV_{g}\,,
\end{align*}
where $C(n)$ is defined as in Lemma~\ref{lem-tachi}. Using H\"older inequality and~\eqref{sobolev} with $u:=|W|$, we get
\begin{align*}
0 \geq&\, \int_{M}|\nabla|W||^{2}\,dV_{g} +\frac{2}{n}\int_{M}R|W|^{2}\,dV_{g} -2C(n)\left(\int_{M}|W|^{\frac{n}{2}}\,dV_{g}\right)^{\frac{2}{n}}\left(\int_{M}|W|^{\frac{2n}{n-2}}\,dV_{g}\right)^{\frac{n-2}{n}}\\
\geq & \int_{M}|\nabla|W||^{2}\,dV_{g} +\frac{2}{n}\int_{M}R|W|^{2}\,dV_{g} \\
&\,-2\frac{C(n)}{Y(M,[g])}\left(\int_{M}|W|^{\frac{n}{2}}\,dV_{g}\right)^{\frac{2}{n}}\left( \frac{4(n-1)}{n-2}\int_{M}|\nabla|W||^{2}\,dV_{g} +\int_{M}R|W|^{2}\,dV_{g}\right)
\end{align*}
Now by assumption one has
$$
\left(\int_{M}|W|^{\frac{n}{2}}\,dV_{g}\right)^{\frac{2}{n}} \,<\, A(n)\,Y(M,[g]) \,.
$$
It then follows that $|W|\equiv 0$, so $g$ has constant positive sectional curvature, if $A(n)$ satisfies the following inequalities
\begin{equation*}
\begin{cases}
2C(n)A(n) \,\leq\, \frac{n-2}{4(n-1)} \\
2C(n)A(n) \,\leq\, \frac{2}{n} \,.
\end{cases}
\end{equation*}
Since  $C(5)=1$, $C(6)=\frac{\sqrt{70}}{2\sqrt{3}}$ and $C(n)=\frac{5}{2}$ for $n\geq 7$, it is easy to see that we can take , $A(5)=\frac{3}{32}$, $A(6)=\frac{\sqrt{3}}{5\sqrt{70}}$, $A(n)=\frac{n-2}{20(n-1)}$ if $7\leq n \leq 9$ and $A(n)=\frac{2}{5n}$ if $n\geq 10$. This concludes the proof of the theorem when $n\neq 4$. 

In dimension $n=4$ we can improve the estimate by using the following refined Kato inequality proved in~\cite{gurleb} (see also~\cite[Lemma 4.3]{bour}) which holds on every four-dimensional Riemannian manifold with harmonic Weyl tensor (in particular if $g$ is Einstein)
$$
|\nabla W|^{2} \, \geq \, \frac{5}{3} |\nabla|W||^{2} \,.
$$
Hence reasoning as before, we have
\begin{align*}
0 \geq&\, \frac{5}{3}\int_{M}|\nabla|W||^{2}\,dV_{g} +\frac{1}{2}\int_{M}R|W|^{2}\,dV_{g} \\
&\,-2\frac{C(4)}{Y(M,[g])}\left(\int_{M}|W|^{2}\,dV_{g}\right)^{\frac{1}{2}}\left( 6\int_{M}|\nabla|W||^{2}\,dV_{g} +\int_{M}R|W|^{2}\,dV_{g}\right)
\end{align*}
Hence, we need
\begin{equation*}
\begin{cases}
2C(4)A(4) \,\leq\, \frac{5}{18} \\
2C(4)A(4) \,\leq\, \frac{1}{2} \,.
\end{cases}
\end{equation*}
Since $C(4)=\frac{\sqrt{6}}{4}$ we obtain $A(4)=\frac{5}{9\sqrt{6}}$.
\end{proof}

\begin{rem} As already observed in the introduction, this result was proved in~\cite{hebvau1} with a slightly stronger pinching assumption in dimension $4\leq n \leq 6$. In particular, in dimension $n=5$, their theorem does not apply under condition~\eqref{pinchein}.
\end{rem}

\

\section{Four dimensional shrinking Ricci solitons} \label{sec-4d}

In this section we will prove Corollary~\ref{cor-4d2} and Remark~\ref{rem1}. Let $(M^{4},g)$ be a compact four-dimensional Riemannian manifold. First of all we recall the Chern-Gauss-Bonnet formula (see~\cite[Equation 6.31]{besse})
\begin{equation}\label{cgb}
\int_{M}|W|^{2}\,dV_{g}-2\int_{M}|\rd|^{2}\,dV_{g}+\frac{1}{6}\int_{M}R^{2}\,dV_{g} \,=\, 32\pi^{2}\chi(M) \,.
\end{equation}
If we denote with $\sigma_{2}(A)$ the second-elementary function of the eigenvalues of the Schouten tensor $A:=\frac{1}{2}\left(Ric-\frac{1}{6}R\,g\right)$, it is easy to see that 
$$
\sigma_{2}(A) \,=\, \frac{1}{96}R^{2} - \frac{1}{8}|\rd|^{2}
$$
and the Chern-Gauss-Bonnet formula reads
$$
\int_{M}|W|^{2}\,dV_{g}+16\int_{M}\sigma_{2}(A)\,dV_{g} \,=\, 32\pi^{2}\chi(M) \,.
$$
By the conformal invariance of the $L^{2}$-norm of the Weyl tensor in dimension four, it follows that the integral of $\sigma_{2}(A)$ is conformally invariant too. To prove Corollary~\ref{cor-4d2}, we need the following lower bound for the Yamabe invariant which was proved by M. J. Gursky~\cite{gursky1}. Since the proof is short, for the sake of completeness, we include it.
\begin{lemma}\label{lem-gur} Let $(M^{4},g)$ be a compact four-dimensional manifold. Then, the following estimate holds
$$
Y(M,[g])^{2} \,\geq\, 96\int_{M}\sigma_{2}(A)\,dV_{g} \,=\,\int_{M}R^{2}\,dV_{g} - 12\int_{M}|\rd|^{2}\,dV_{g}\,.
$$
Moreover, the inequality is strict unless $(M^{4},g)$ is conformally Einstein.
\end{lemma}
\begin{proof} We solve the Yamabe problem in $[g]$. Let $\gt\in[g]$ be a Yamabe metric. Then
\begin{align*}
Y(M,[g])^{2} =& \frac{\left(\int_{M}\widetilde{R}\,dV_{\gt}\right)^{2}}{\int_{M}\,dV_{\gt}} \,=\, \int_{M}\widetilde{R}^{2}\,dV_{\gt} \\
\geq& \int_{M}\widetilde{R}^{2}\,dV_{\gt} -12\int_{M}|\widetilde{\rd}|^{2}\,dV_{\gt} \\
=&\, 96 \int_{M} \sigma_{2}(\widetilde{A})\,dV_{\gt} \,=\, 96 \int_{M}\sigma_{2}(A)\,dV_{g},
\end{align*}
where in the last equality we have used the conformally invariance of $\int_{M}\sigma_{2}(A)\,dV_{g}$. The equality case follows immediately.
\end{proof}

From this lemma we get
\begin{equation*} 
\int_{M}|W|^{2}\,dV_{g}+\int_{M}|\rd|^{2}\,dV_{g} -\frac{1}{48}Y(M,[g])^{2}\,\leq\, \int_{M}|W|^{2}\,dV_{g}+\frac{5}{4}\int_{M}|\rd|^{2}\,dV_{g}-\frac{1}{48}\int_{M}R^{2}\,dV_{g}\,.
\end{equation*}
Moreover, the inequality is strict unless $(M^{4},g)$ is conformally Einstein. In the first case, Theorem~\ref{teo-4d1} immediately implies Corollary~\ref{cor-4d2}. In the second case, the fact that $g$ is conformally Einstein, in particular, implies that $(M^{4},g)$ is Bach flat (see~\cite[Proposition 4.78]{besse}). Since $M^{4}$ is compact, by a result of H.-D. Cao and Q. Chen~\cite{caochen2} it follows that $(M^{4},g)$ has to be Einstein. Following the proof of Theorem~\ref{teo-ndim}, we can see that the pinching condition together with Theorem~\ref{teo-ein} concludes the proof of Corollary~\ref{cor-4d2}.

Finally, observe that by the Chern-Gauss-Bonnet formula~\eqref{cgb}, the right-hand side can be written as
$$
\int_{M}|W|^{2}\,dV_{g}+\frac{5}{4}\int_{M}|\rd|^{2}\,dV_{g}-\frac{1}{48}\int_{M}R^{2}\,dV_{g} \,=\, \frac{13}{8}\int_{M}|W|^{2}\,dV_{g}+\frac{1}{12}\int_{M}R^{2}\,dV_{g}-20\pi^{2}\chi(M) \,.
$$
This proves Remark~\ref{rem1}.

%
%
%
%
%

\

\

\begin{ackn}
\noindent The author is member of the Gruppo Nazionale per
l'Analisi Matematica, la Probabilit\`{a} e le loro Applicazioni (GNAMPA) of the Istituto Nazionale di Alta Matematica (INdAM).
\end{ackn}

\

\

\bibliographystyle{amsplain}
\bibliography{biblio}

\

\

\parindent=0pt

\end{document}